\documentclass[a4, 11pt]{amsart}
\usepackage{amssymb,latexsym}
\usepackage{color}
\usepackage[left=35mm, right=35mm, top=30mm, bottom=30mm, includefoot, includehead]{geometry}
\theoremstyle{plain}
\newtheorem{theorem}{Theorem}[section]
\newtheorem{corollary}[theorem]{Corollary}
\newtheorem{proposition}[theorem]{Proposition}
\newtheorem{lemma}[theorem]{Lemma}
\theoremstyle{definition}
\newtheorem{definition}[theorem]{Definition}
\newtheorem{example}[theorem]{Example}

\newtheorem{remark}[theorem]{Remark}
\newtheorem{question}[theorem]{Question}

\newcommand{\enm}[1]{\ensuremath{#1}}          %

\newcommand{\cal}[1]{\mathcal{#1}}

\newcommand{\CC}{\enm{\mathbb{C}}}

\newcommand{\NN}{\enm{\mathbb{N}}}

\newcommand{\PP}{\enm{\mathbb{P}}}

\newcommand{\JJ}{\enm{\mathbb{J}}}

\newcommand{\Ff}{\enm{\cal{F}}}

\newcommand{\Ii}{\enm{\cal{I}}}
\newcommand{\Jj}{\enm{\cal{J}}}

\newcommand{\Oo}{\enm{\cal{O}}}

\newcommand{\Ss}{\enm{\cal{S}}}

\newcommand{\Uu}{\enm{\cal{U}}}
\newcommand{\Vv}{\enm{\cal{V}}}

\renewcommand{\phi}{\varphi}
\renewcommand{\theta}{\vartheta}
\renewcommand{\epsilon}{\varepsilon}

\DeclareMathOperator{\gen}{gen}
\DeclareMathOperator{\reg}{reg}

\renewcommand{\to}[1][]{\xrightarrow{\ #1\ }}

\newcommand{\old}[1]{}

\date{}
\title[]{Entry loci and ranks}

\author{Edoardo Ballico and Emanuele Ventura}

\address{Universit\`a di Trento, 38123 Povo (TN), Italy}
\email{edoardo.ballico@unitn.it}

\address{Dept. of Mathematics, Texas A\&M University,
College Station, TX 77843-3368, USA}
\email{eventura@math.tamu.edu, emanueleventura.sw@gmail.com}


\begin{document}

\maketitle

\begin{abstract}
We study entry loci of varieties and their irreducibility from the perspective of $X$-ranks with respect to a projective variety $X$. 
These loci are the closures of the points that appear in an $X$-rank decomposition of a general point in the ambient space. We look at 
entry loci of low degree normal surfaces in $\PP^4$ using Segre points of curves; the smooth case was classically studied
by Franchetta. 
We introduce a class of varieties whose generic rank coincides with the one of its general entry locus, and show that any smooth 
and irreducible projective variety admits an embedding with this property. 
\end{abstract}

\section{Introduction}

Given a projective variety $X\subset \PP^r$ and a point $q\in \PP^r$, the $X$-rank of $q$, denoted $\mathrm{rk}_X(q)$, is the minimal number of points  $p_1,\ldots, p_s\in X$ required to span $q$.  

The $X$-rank is a generalization of matrix rank and of the more recently studied tensor and Waring ranks. The recent burst of results on the latter ones has reinvigorated the interest in classical questions linked to them, e.g. the study of tangents, secants and their defectivity issues \cite{r3,z}. However, other objects in classical algebraic geometry may have natural interpretations in terms of ranks: studying them could shed light, provide new interesting insights and perspectives to $X$-ranks.  This is the case for the entry locus of a general point in the ambient space of $X$. 

Let $X\subset \PP^r$ be a projective variety and $q\in \PP^r$ a general point; the entry locus $\Gamma_q(X)$ of $q$ is the closure of the set of all points $p\in X$ appearing in an $X$-rank decomposition of $q$. We believe that a better understanding 
of entry loci would be a valuable source of information for the embedded projective geometry of $X$ and relevant for a deeper analysis of rank questions. 

The entry locus is a natural object associated to decompositions of $q$, and so naturally related to varieties of sums of powers ($\mathrm{VSP}$'s, for short) of $q$. The latter ones have attracted a lot of attention since the work of Mukai \cite{Mukai1, Mukai2}. Thereafter, several authors gave more instances of remarkable geometric nature, see e.g. \cite{rs}.

When $\mathrm{rk}_X(q)=2$, the corresponding entry loci have been already shown to play a significant role in reflecting important projective properties of $X$. Ionescu and Russo \cite{r2,ir} studied smooth varieties with ({\it local}) {\it quadratic entry locus} (the so-called $QEL$ manifolds and their local version, the $LQEL$ manifolds) and established several results about them. It is an interesting class of varieties; for instance, they are rationally connected and have extremal tangential behavior. 

We propose the possibly daunting study of higher entry loci, i.e. $\mathrm{rk}_X(q)>2$. One of our main long-term goals is to determine whether remarkable properties and structures appear for entry loci as well. This work is a first attempt to that. Similarly to the context of $\mathrm{VSP}$'s, there is a lack of explicit examples and methods to attack general situations. For instance, one of the basic questions we wonder is what conditions are favorable in order to detect irreducibility of entry loci (Question \ref{q1}); we provide an answer in a specific situation in  \S\ref{irredentryloci}. This is a very classical issue and work of Lopez and Ran \cite{lr} (in the case of the entry locus $\Gamma_q(X)$ when $\mathrm{rk}_X(q)=2$) proved irreducibility in some range; in {\it loc. cit.} this result was shown to imply a variant of Zak's linear normality theorem \cite[Theorem II.2.14]{z}. In these regards, we wonder what is the precise relation (if any) between
higher entry loci and a generalization of Zak's linear normality theorem proved by Chiantini and Ciliberto \cite[Theorem 7.5 and Corollary 7.6]{cc3}. 

For the sake of comparison, the very same irreducibility question may be asked for $\mathrm{VSP}$'s; as far as we know, several known results for these varieties (e.g. the rationality results of Massarenti and Mella \cite{mm, ma}) apply to each of the irreducible components, without knowing the irreducibility to begin with. 

Another geometric tool we hope to bring to problems in the setting of $X$-ranks is the interesting notion of {\it Segre points} of a projective variety, studied by Calabri and Ciliberto \cite{cac} and introduced by B. Segre. We specifically look at Segre points of entry loci and derive some geometric information: we apply them to the first intriguing case of entry loci, i.e. that of normal surfaces in $\PP^4$. For the latter surfaces, we focus on the first low degree cases. In the smooth case, for completeness, we include a proof of the irreducibility of entry loci (with essentially one exception), which follows from the classical work on the {\it double curve} (``la curva doppia'') by Franchetta dating back to 1941 \cite{fra}. \\

\noindent {\bf Contributions and structure of the paper.} We now describe more thoroughly the content of this paper. Notation and preliminary definitions used throughout the article are set up in \S\ref{prelim}. Here we also record some basic facts about entry loci. 

In \S\ref{segrepoints}, we recall the definition of Segre points for a projective variety and we prove Lemma \ref{y5} and Lemma \ref{y9} that are useful
to describe entry loci for some surfaces in $\PP^4$, which is achieved in \S\ref{surfaces}. In Theorem \ref{uu0}, we give a description of the closure of
the Segre locus of the entry locus of $n$-dimensional projective variety $X\subset \PP^{2n}$. 

The delicate question raised in \S\ref{irredentryloci} asks how to detect irreducibility of the general entry loci $\Gamma_q(X)$. The answer to such a classical 
question seems to be out of reach in general. In some range, this problem was solved by Lopez and Ran \cite{lr}.

Theorem \ref{j05} and Proposition \ref{05} provide such a case when $X$ is a birational  general linear projection. 
Next, we introduce {\it types} (Definition \ref{types}) according to whether the entry locus $\Gamma_q(X)$ is irreducible ({\it type} {\bf I}) or not  ({\it type} {\bf II}). 
Moreover, $X$ is said to be of {\it type} {\bf A} if for a general point in the span $\langle \Gamma_q(X)\rangle$ the entry locus does not change (Definition \ref{typesa&b}). This turns out to be a valuable property in terms of ranks with respect to $X$ and $\Gamma_q(X)$, as highlighted by Lemma \ref{typeAgen} and Theorem \ref{x5}. The latter relates {\it type} {\bf IA} with $QEL$-manifolds of Ionescu and Russo \cite{ir,r2}. Remark \ref{x6} shows that any smooth complete intersection $X\subset \PP^{2n}$ of $n$ hypersurfaces of degrees $d_1,\ldots, d_{n}\geq 2$ such that $d_1+\cdots +d_{n}\ge 2n+1$ is of type {\bf B}. 

In \S\ref{surfaces}, we look at normal surfaces of low degree $X\subset \PP^4$. When $X$ is smooth, we record a proof of the irreducibility of the entry loci for all smooth surfaces $X\subset \PP^4$, unless $X$ is an isomorphic projection of the Veronese surface in $\PP^5$, which 
is derived from a classical result due to Franchetta \cite{fra}. Although we believe this is known to experts, we could not find a reference
where the proof is written explicitly; so we include a proof for the sake of completeness. 

For such surfaces $X$, the general entry locus is a curve. In Proposition \ref{degentryloc}, we give the degree of $\Gamma_q(X)$. Corollary \ref{uuuu1} describes
the situation for $X$ being a normal cone. Proposition \ref{mindegsurf} provides a description of $\Gamma_q(X)$ for normal surfaces of minimal degree three. 
In Theorem \ref{surfdegfour}, we analyze entry loci of degree four surfaces with sectional genus one.  

\S\ref{afamily} is devoted to the class $\mathcal F(n,s)$ of projective varieties whose entry loci are irreducible curves (Definition \ref{largeclass}). We prove that
any smooth and irreducible projective variety $W$ can be embedded in such a way $W$ is isomorphic to an element of $\mathcal F(n,s)$; in fact, Theorem \ref{u1} gives sufficient conditions on a line bundle in order to define such an embedding. As a consequence, perhaps surprisingly, this result suggests how large the collection $\mathcal F(n,s)$ is.

In \S\ref{tangdeg}, under the assumption $(\dagger)$ (i.e. in characteristic zero there are no {\it tangentially degenerate curves}), we prove that for a {\it type} {\bf A} non-defective $n$-dimensional variety $X\subset \PP^{s(n+1)-2}$, its entry locus $\Gamma_q(X)$ is the rational normal curve of degree $2s-2$ (Theorem \ref{v1}). In the special case when $s=3$, {\it without assuming} $(\dagger)$, Theorem \ref{d2} shows that $\Gamma_q(X)$ is rational. 
 \vspace{3mm}

\begin{small}
\noindent {\bf Acknowledgements.} The first author was partially supported by MIUR and GNSAGA of INdAM (Italy). The authors would like
to thank Francesco Russo for useful discussions. 
\end{small}

\section{Preliminaries}\label{prelim}
We work over an algebraically closed field of characteristic zero. 

Let  $X_1,\dots ,X_s\subset \PP^r$ be integral and non-degenerate varieties. Assume $1\le s\le r+1$, so that $\{p_1,\dots ,p_s\}$ is linearly independent for a general collection $(p_1,\dots , p_s)\in X_1\times \cdots \times X_s$. (Although, the assumption on $s$ is not necessary in what follows.) 

The {\it abstract join} $\Jj(X_1,\dots ,X_s)\subset X_1\times \cdots \times X_s\times \PP^r$ is the closure of the locally closed set of all tuples $(p_1,\dots ,p_s, q)$ with $p_i\in X_i$, such that $q\in \langle p_1,\dots ,p_s\rangle$. Let $\pi: \Jj(X_1,\dots ,X_s)\to \PP^r$ be the projection onto the last factor $\PP^r$. The projective variety $\JJ(X_1,\dots ,X_s):= \pi (\Jj(X_1,\dots ,X_s))\subseteq \PP^N$ is the {\it join} of $X_1,\dots ,X_s$. When $X = X_i$, the (abstract) join is called ({\it abstract}) $s$-th {\it secant} of $X$. Let $\sigma_s(X)$ denote the $s$-th secant of $X$; $r_{\gen}(X)$ is the least integer such that $\sigma_{r_{\gen}}(X) = \PP^r$. 

Joins and secants can be naturally defined for reducible varieties as well \cite[\S 1]{cc3}; in this case they might be reducible. 

 A finite set $S\subset \PP^r$  \emph{irredundantly spans} $q\in \PP^r$ if $q\in \langle S\rangle$ and $q\notin \langle S'\rangle$ for any $S'\subsetneq S$. The cardinality of a finite set $S$ is $\sharp(S)$. 

\begin{definition}

Let $X\subset \PP^r$ be a non-degenerate projective variety. The $X$-rank of a point $q\in \PP^r$ is 
\[
r_X(q) = \min\lbrace s \ | \ q\in \langle p_1,\ldots, p_s \rangle, p_i\in X \rbrace, 
\]
i.e. the finite set $\lbrace p_1,\ldots, p_{r_X(q)}\rbrace$ irredundantly spans $q$. A general $q\in \PP^r$ satisfies $r_X(q) = r_{\gen}(X)$. Whenever $X$ is degenerate in its ambient space, the definition applies to the points $q\in \langle X\rangle$. 
\end{definition}

\begin{definition}
Let $Y$ be an equidimensional projective variety of dimension $m$. The $s$-th secant of $Y$ has {\it the expected dimension} if $\dim \sigma_s(Y) = \min \{(m+1)s-1,\dim\langle Y\rangle\}$. When all secants of $Y$ have the expected dimension, $Y$ is said to be {\it non-defective}. 
\end{definition}

\begin{definition}\label{entrylocus}
For any $q\in \PP^r$, let $\Ss (X_1,\dots ,X_s,q)$ denote the set of all distinct $\{p_1,\dots ,p_s\}\subset \PP^r$ such that $p_i\in X_i$ and $\{p_1,\dots ,p_s\}$ irredundantly spans $q$. Denote the closure of  $\Ss (X_1,\dots ,X_s,q)$ in $\PP^r$ by $\Gamma _q(X_1,\dots ,X_s)$. This locus is called the {\it entry locus} of $X_1,\ldots, X_s$ at $q$. When $X = X_i$, the entry locus is denoted $\Gamma_q(X, s)$ and called the {\it $s$-th entry locus} of $X$. 

When $s=r_{\gen}(X)$, the entry locus is simply denoted $\Gamma_q(X)$. Moreover, define:
\[
\gamma(X) = \dim \Gamma _q(X) \mbox{ and } \ell(X) = \dim \langle \Gamma _q(X)\rangle.
\] 
\end{definition}

\begin{remark}
In general, the $s$-th {\it entry locus} of $q$ is not irreducible \cite[p. 96]{r3}, or \cite[Example 3.18]{cc3} which describes the example of a scroll.
\end{remark}

\begin{remark}
The $s$-th entry locus $\Gamma_q(X)$ is contained in the tangential $(s+1)$-contact locus of any $s+1$ points in $\Gamma_q(X)$;
see \cite[Definition 3.4]{cc3} and \cite[Remark 3.14]{cc3}. 
\end{remark}

The next lemma records the dimension of $\Gamma_q(X)$; see \cite[Definition 1.4.6]{r3} or \cite[Proposition 3.13]{cc3}. 

\begin{lemma}\label{dimentrylocus}
Let $X\subset \PP^r$ be integral and non-degenerate variety. Then: 
\[
\gamma(X) = \dim \Gamma _q(X) = \dim\sigma_{r_{\gen}(X)-1}(X) + \dim X + 1-r.
\]
\end{lemma}

\begin{remark}\label{x1}
The entry loci $\Gamma _q(X)$ for $q$ general are equidimensional \cite[p.21]{r3}. 
\end{remark}

\section{Entry loci and Segre points}\label{segrepoints}

In this section, we discuss Segre points following \cite{cac} and the classical work of B. Segre. Results from here
are only used in \S\ref{surfaces}, and they might be of independent interest. 

\begin{definition}[{\bf \cite[\S 1]{cac}}]

Let $Y\subset \PP^r$ be an integral projective variety. A {\it Segre point} of $Y$ is a point $q\in \PP^r\setminus Y$ such that 
the projection away from $q$ induces a {\it non-birational} map on $Y$. The set of Segre points of $Y$ is denoted $\Sigma(Y)$. 
This is a constructible set and its closure $\mathfrak{S}(Y)$ is the {\it Segre locus} of $Y$. 
\end{definition}

This is generalized to pairs of the same dimension as follows: 

\begin{definition}
Let $Y, T\subset \PP^r$ be distinct integral projective varieties of the same dimension and such that $\langle Y\cup T\rangle
=\PP^r$. A {\it Segre point} $q$ of $Y\cup T$ is a point $q\in \PP^r\setminus (Y\cup T)$ such that $\pi _q(Y) =\pi _q(T)$,
i.e. a general line through $q$ meeting $Y$ meets $T$ as well.  Let $\Sigma (Y,T)$ denote the set of all Segre points of $Y\cup T$. 
Let $\mathfrak{S}(Y,T)$ denote the closure of $\Sigma (Y,T)$ in $\PP^r$.  We say that $\mathfrak{S}(Y,T)$ is the {\it Segre locus} of $Y\cup T$, or of the pair $(Y,T)$. 
\end{definition}

\begin{remark}\label{y3}
Let $Y, T\subset \PP^r$ be integral curves such that $Y\ne T$ and $\langle Y\cup T\rangle =\PP^r$.
\begin{enumerate}

\item[(i)] If $r=2$, then $\Sigma (Y,T) =\PP^2\setminus (Y\cup T)$.

\item[(ii)] Assume $Y\cap T=\emptyset$ and that $Y$ and $T$ are lines. Thus $r=3$. For each $o\in \PP^3\setminus Y\cup T$ there
is a unique line meeting both $Y$ and $T$. Hence $\Sigma (Y,T)=\emptyset$.

\item[(iii)] Assume that $Y$ is a line and that  $T$ is not a line. In this case, $\Sigma (Y,T)\ne \emptyset$ if and only if
$r=2$.

\item[(iv)] Assume $\dim \langle Y\rangle =\dim \langle T\rangle =2$ and $r>2$. Then it is clear that $3\le r\le 5$. 
\vspace{2mm}
\begin{quote}
\noindent {\it Claim}: If $\Sigma (Y,T) \ne \emptyset$ then $r=3$ and there is an isomorphism $u: \langle Y\rangle \to \langle T\rangle$ mapping $Y$ onto $T$ and fixing point-wise the line $\langle Y\rangle \cap \langle T\rangle$. \\

\noindent {\it Proof of the Claim}: Let $o\in \Sigma (Y,T)$ and call $W$ the cone with base $Y$ and $o$ as its vertex. Since $o\in \Sigma (Y,T)$, we have $T\subset W$. Hence $\langle W\rangle \supseteq \langle Y\cup T\rangle =\PP^r$. Thus $\langle W\rangle =\PP^r$. Since $r>2$, we have $W\nsubseteq \langle Y\rangle$. Therefore $o\notin \langle Y\rangle$. Since $\langle Y\rangle$ is a plane, we get $r=3$. Since $\Sigma (Y,T) =\Sigma (T,Y)$,  one has $o\notin \langle T\rangle$. Hence the linear projection from $o$ induces an isomorphism $ u: \langle Y\rangle \to \langle T\rangle$ mapping $Y$ onto $T$ and fixing point-wise the line $\langle Y\rangle \cap \langle T\rangle$.
\end{quote}
\vspace{2mm}
\item[(v)] Assume $\dim \langle Y\rangle >\dim \langle T\rangle =2$. If $o\in \langle T\rangle$ and $o\notin Y\cup T$, then
$o\notin \Sigma (Y,T)$. 

If $o\notin \langle Y\rangle$, then $o\notin \Sigma (Y,T)$. Indeed, if $o\in \Sigma(Y,T)$, the projection from $o$ would map $Y$ 
isomorphically to the curve $T$ whose span would have the same dimension, a contradiction. 

If $r>3$, then $\Sigma (Y,T)=\emptyset$. On the contrary, let $W$ be the cone with vertex $o\in \Sigma(Y,T)$ and $T$ as its base. Thus $\langle W\rangle$ has dimension three and contains $\langle T\rangle$; however, $W$ contains $Y$, by definition of the point $o$, and hence $\langle W\rangle$ contains $\langle Y\rangle$ as well. This implies $\langle W\rangle = \langle Y\rangle\supset \langle T\rangle$ and so $r=3$, a contradiction. 

Assume $r=3$ and take $o\in \PP^3\setminus (Y\cup \langle T\rangle)$. If $\deg (Y)<\deg (T)$, then $o\notin \Sigma
(Y,T)$, as otherwise the projection from $o$ would induce an isomorphism from $T$ onto $\pi_o(T)$ and $\deg(T) = \pi_o(T) = \pi_o(Y)\leq \deg(Y)$. On the other hand, if $\deg (Y)>\deg (T)$ and $o\in \Sigma (Y,T)$, then $o$ is a Segre point of $Y$ and hence there are at most finitely
many such points \cite[Theorem 1]{cac}. 

Assume $\deg (Y)=\deg (T)$. In this case, $\Sigma (Y,T)\ne \emptyset$ if and only if the
normalization $u: \tilde{T}\to T$ factors through morphisms $u': \tilde{T}\to Y$ and $v: Y\to T$ (i.e., $u=v\circ u'$ with $v$ induced by
a linear projection from a Segre point $o\in \Sigma (Y,T)$).

\end{enumerate}
\end{remark}

\begin{remark}\label{y4}
Let $r>3$. Let $Y\subset \PP^r$ and $T\subset \PP^r$ be integral curves such that $Y\ne T$ and $\langle Y\cup T\rangle =\PP^r$.
Let $V\subset \PP^r$ be a general $(r-4)$-dimensional linear space. Let $\pi_V : \PP^r\setminus V\to \PP^3$ denote the linear
projection from $V$. We have $Y\cap V=T\cap V = \emptyset$, $\dim \langle \pi_V (Y)\rangle =\min \{3,\dim \langle Y\rangle\}$
and $\dim \langle \pi_V (T)\rangle =\min \{3,\dim \langle T\rangle\}$. Since $V$ is general, it contains no irreducible
component $D$ of $\mathfrak{S}(Y,T)$ and $\dim \pi_V (D\setminus D\cap V) =\min \{3,\dim D\}$. We have $\pi_V (\Sigma
(Y,T)\setminus \Sigma (Y,T)\cap V)\subseteq \Sigma(\pi_V(Y),\pi_V(T))$.
\end{remark}

\begin{lemma}\label{y5}
Let $Y, T\subset \PP^r$, $r\ge 3$, be integral curve. Assume $Y\ne T$ and $\langle Y\cup T\rangle =\PP^r$.
Then $\dim \Sigma (Y,T)\le 1$. If $\dim \Sigma (Y,T)=1$, then $\deg (Y)=\deg (T)$ and for any
one-dimensional irreducible component $D$ of $\mathfrak{S}(Y,T)$ we have
$\deg (D)=\deg (Y)$. Moreover, the curves $Y$, $T$, and $D$ are birational.
\end{lemma}

\begin{proof}
Let $d = \deg (Y)$. Taking a general linear projection, we reduce to the case $r=3$, by Remark \ref{y4}. Take any irreducible
curve $D'\subseteq \Sigma (Y,T)$ and let $D\subseteq \mathfrak{S}(Y,T)$ its closure in $\PP^3$. By Remark \ref{y3},
we may assume that $\langle Y\rangle =\langle T\rangle =\PP^3$. For a general
$o\in D$, each line through $o$ meeting $Y$ meets $T$. Since $Y$ and $T$ have only finitely many Segre points \cite[Theorem 1]{cac}, we get that for a general $o\in Y$ the curve $\pi _o(Y)$ (resp. $\pi _o(T)$) has degree $d$ (resp. $\deg (T)$). Since $\pi_o(Y) =\pi_o(T)$, one has
$\deg (T) =d$, and that $Y$ and $D$ are birational. 

There are $\infty ^2$-many lines spanned by a point of $Y$ and a different point of $T$. Since
$\dim D=1$ and a general $o\in D$ is contained in $\infty ^1$-many lines meeting $Y$ and a different point of $T$, we obtain
$\mathfrak{S}(D,Y)\supseteq T$. Thus, by the same arguments as above, $\deg (D) =d$, and $D$ and $Y$ are birational.
\end{proof}

\begin{theorem}\label{uu0}
Let $n\geq 2$ and $X\subset \PP^{2n}$ be an $n$-dimensional integral and non-degenerate variety such that $\sigma _2(X)=\PP^{2n}$.
Let $q\in \PP^{2n}$ be general. Then:
\begin{enumerate}

\item[(i)] $\dim \mathfrak{S}(\Gamma_q(X))\le 2$.

\item[(ii)] Let $\Pi$ be any two-dimensional irreducible component of $\mathfrak{S}(\Gamma_q(X))$. Then $\Pi$
is a plane spanned by the union of the components of $\Gamma _q(X)$ contained in $\Pi$.
\end{enumerate}
\begin{proof}

(ii). Take an irreducible component $\Delta$ of $\mathfrak{S}(\Gamma_q(X))$
such that $\dim \Delta \geq 2$. By Lemma \ref{dimentrylocus}, $\dim
\Gamma _q(X)=1$. 

Now, either there is an irreducible component $D$ of $\Gamma_q(X)$ such that $\mathfrak{S}(D) \supseteq \Delta$ or there are irreducible components $Y$ and $T$ of the curve $\Gamma_q(X)$ such that $\mathfrak{S}(Y,T) \supseteq \Delta$.

Assume the existence of $D$ and let $m = \dim \langle D\rangle\geq 2$. If $m\ge 3$, then $\dim \mathfrak{S}(D) \le 0$
by \cite[Theorem 1]{cac}. Thus $m=2$. The plane $\langle D\rangle$ is by definition spanned by $D$ and coincides with $\mathfrak{S}(D)$. Therefore $\mathfrak{S}(D) = \Delta$. So $\Delta$ is a plane. 

Assume the existence of $Y$ and $T$. By Lemma \ref{y5}, we have $\dim \langle Y\cup T\rangle =2$. Thus
$\mathfrak{S}(Y,T) =\langle Y\cup T\rangle$, as in Remark \ref{y3}(i). As before, $\mathfrak{S}(Y,T) = \Delta$ and hence $\Delta$ is a plane. Moreover, the analysis above implies (i). \end{proof}
\end{theorem}

\begin{lemma}\label{y9}
Let $Y\subset \PP^r$ and $T\subset \PP^r$ be integral projective varieties such that $\langle Y\cup T\rangle = \PP^r$ and $M = \langle Y\rangle$ has dimension $m\le r-2$. Then there is no $o\in \PP^r\setminus (M\cup T)$ such that the linear projection $\pi _o: \PP^r\setminus \{o\}\to \PP^{r-1}$ satisfies $\pi _o(T) \subseteq \pi _o(Y)$.
\end{lemma}

\begin{proof}
Assume that such an $o$ exists, i.e. assume that the cone $W$ with vertex $o$ and base $Y$ contains $T$. Since $W$ contains $Y\cup T$, we
have $\langle W\rangle =\PP^r$. However, $\langle W\rangle \subseteq \langle \{o\}\cup M\rangle$, and so $\dim \langle W\rangle  \le m+1<r$, a contradiction.
\end{proof}

We include the next result, which is not used anywhere else in the rest; however it may be of interest on its own right:

\begin{lemma}\label{y10}
Let $Y\subset \PP^r$ and $T\subset \PP^r$ be integral projective varieties such that $\langle Y\cup T\rangle = \PP^r$ and $\dim Y=\dim T$. Let $M = \langle Y\rangle$, $H = \langle T\rangle$, $m = \dim M$
and $h = \dim H$. Assume $m<r$, $h=r$ and $\deg (T) < 2\deg (Y)$. 
Then there is no $o\in \PP^r\setminus (M\cup T)$ such that the linear projection $\pi _o: \PP^r\setminus \{o\}\to \PP^{r-1}$ satisfies $\pi _o(T)\subseteq  \pi _o(Y)$.
\end{lemma}

\begin{proof}
 By Lemma \ref{y9}, we may assume $m=r-1$. Assume that $o$ exists, i.e. assume that the cone $W$ with vertex $o$ and base $Y$ contains $T$. Since $o\notin M$, each line of $W$ containing $o$ intersects transversally $Y$ at a unique point. Hence $\pi_o$ induces a birational map on $Y$ onto its image; one has
 $\deg(\pi_o(Y)) = \deg(Y)$. 
 Since $o\notin T$, $\pi_{o|T}: T\to \PP^{r-1}$ is a finite morphism. Since $\dim Y = \dim T$, one has $\dim \pi _o(T) =\dim Y$ and hence $\pi _o(T) =\pi _o(Y)$. Then the map $\pi_{o|T}$ can be regarded as a finite morphism $\psi: T\rightarrow \pi_o(Y)$. Thus 
 \[
 \deg(T) = \deg(\psi)\deg(\psi(T)) = \deg(\psi)\deg(\pi_o(Y)) = \deg(\psi)\deg(Y).
 \]
By assumption $\deg (T) <  2\deg (Y)$, and hence $\psi$ is a birational morphism. In conclusion, $\deg (T) =\deg (Y) = \deg (W)$. However, the irreducible cone $W$ has no non-degenerate codimension one subvariety of degree $\deg (W)$ (the only ones with this degree being the hyperplane sections), a contradiction.
\end{proof}

\section{Irreducibility of some entry loci}\label{irredentryloci}
One of the first questions one might ask on entry loci is as follows: 

\begin{question}\label{q1}
Let $X\subset \PP^r$. When is $\Gamma _q(X)$ irreducible?
\end{question}

\begin{theorem}\label{j05}
Fix integers $s\ge 2$, $r, m\ge 1$, and integral non-degenerate varieties $X'_i\subset \PP^{r+m}$ such that $\JJ(X'_1,\dots ,X'_s) =\PP^{r+m}$ and $\dim X'_i \le r$ for all $i$. Let $E\subset \PP^{r+m}$
be a general $(m-1)$-dimensional linear subspace. Let $\pi _E:\PP^{r+m}\setminus E\to \PP^r$ denote the linear projection from $E$. Set $X_i:= \pi _E(X'_i)$. Then $\Gamma _q(X_1,\dots ,X_s)$ is irreducible for a general $q\in \PP^r$.
\end{theorem}
\begin{proof}
Since $\JJ(X'_1,\dots ,X'_s) =\PP^{r+m}$, we have $\JJ(X_1,\dots ,X_s) =\PP^r$. Thus, one has $\Ss (X_1,\dots ,X_s,q)\ne \emptyset$ for a general $q\in \PP^r$. Fix a general $q\in \PP^r$.
Let $M\subset \PP^{r+m}$ be the $m$-dimensional linear space containing $E$ and with $\pi _E(M\setminus E) = \{q\}$. Since we first fixed $X'_1,\dots ,X'_s$ and then chose a general subspace $E$ and a general point $q\in \PP^r$, $M$ is a general $m$-dimensional linear subspace of $\PP^{r+m}$. Since $m\ge 1$, $\JJ(X'_1,\dots ,X'_s)=\PP^{r+m}$, $\Jj(X'_1,\dots ,X'_s)$ is irreducible and $M$ is general, applying $r$ times the second Bertini theorem (\cite[part 4) of Th. 6.3]{j} or \cite[Ex. III.11.3]{h}) gives that $\mathrm{proj}^{-1}(M)$ is irreducible, where $\mathrm{proj}$ is the projection onto $\PP^{r+m}$ of the abstract join $\mathcal J(X'_1,\ldots, X'_s)$. The set $\mathrm{proj}^{-1}(M)$ is the closure of the union of all sets $(p_1,\dots ,p_s,o)$ with $o\in M$, $p_i\in X'_i$ for all $i$, $\sharp (\{p_1,\dots ,p_s\})=s$ and $\{p_1,\dots ,p_s\}$ irredundantly spanning $o$. Since $\pi _E(M\setminus E) =\{q\}$, we have a surjection $\mathrm{proj}^{-1}(M)\to \Gamma _q(X_1,\dots ,X_s)$. Thus $\Gamma _q(X_1,\dots ,X_s)$ is irreducible.
\end{proof}

\begin{proposition}\label{05}
Fix an integer $s\ge 2$, $m>0$, and an integral non-degenerate variety $X'\subset \PP^{r+m}$ such that $\sigma _s(X') =\PP^{r+m}$ and $X'$ has codimension at least $m$. Let $X\subset \PP^r$ be a general projection of $X'$ from a general $(m-1)$-dimensional linear space. Then $\Gamma _q(X)$ is irreducible for a general $q\in \PP^r$.
\end{proposition}
\begin{proof}
This is the case when $X'_i =X'$ for all $i$ in Theorem \ref{j05}. 
\end{proof}

\begin{remark}
One would like to extend this to the case when $X$ is not a birational general linear projection. However, there are technical difficulties and counterexamples if we do not assume that $\sigma _s(X') =\PP^{r+m}$. The irreducibility of the entry locus for $s=2$ and $X$ smooth is the main result of Lopez and Ran \cite[Theorem 1]{lr}. 
\end{remark}

\begin{definition}[{\bf Types}]\label{types}
Let $s\in \NN$ and let $X$ be a projective variety. 
The variety $X$ is said to be of {\it type} {\bf I} at $s$ if $\Gamma _q(X,s)$ is irreducible for a general $q\in \PP^r$. Otherwise, $X$ 
is of {\it type} {\bf II} at $s$. When $s =r_{\gen}(X)$ we simply say that $X$ is of {\it type} {\bf I} or {\it type} {\bf II}, respectively. 
\end{definition}

\begin{definition}\label{typesa&b}
Let $X\subset \PP^r$ be integral and non-degenerate variety. The variety $X$ is said to be of {\it type} {\bf A}
if $\Gamma _o(X) =\Gamma _q(X)$ for a general $o\in \langle \Gamma _q(X)\rangle$. Otherwise, $X$ is said to be of {\it type} {\bf B}.
\end{definition}

\begin{remark}\label{existenceofopen}
There is a flattening stratification of the algebraic family of reduced projective schemes $\Gamma _q(X)$, $q\in \PP^r\setminus \sigma _{s-1}(X)$ (\cite[Lecture 8]{mum} or \cite[Theorem 5.13]{nit}), i.e. there is a partition of $\PP^r\setminus \sigma _{s-1}(X)$ into finitely many locally closed subsets $E_1,\dots ,E_t$, such that, for each $i\in \lbrace1,\dots , t\rbrace$, the algebraic family $\Gamma _q(X)$ over $E_i$ is flat. Hence, for all $q\in E_i$, the schemes $\Gamma _q(X)$ have the same Hilbert polynomial \cite[III.9.9]{h}. Call $U$ the unique $E_i$ containing a non-empty open subset of the irreducible variety $\PP^r\setminus \sigma _{s-1}(X)$. Taking a non-empty open subset $V$ of $U$ we may also assume that, for all $q\in V$, the schemes $\Gamma_q(X)$ have the same number of irreducible components and the same Hilbert function. Note that, in particular, $\Gamma_q(X)$ (for $q\in V$) share the spanning dimension, $\ell(X):= \dim \langle \Gamma _q(X)\rangle$. 

Fix $q\in V$ and let $\Uu_q = \lbrace o\in V \ | \ \Gamma_o(X) = \Gamma_q(X)\rbrace$. Since $\Uu _q\subset \langle \Gamma _q(X)\rangle$, the semicontinuity theorem for the local dimensions of fibers of morphisms \cite[Ex. II.3.22(b)]{h} shows that $X$ is of {\it type} {\bf A} if and only if there
exists a non-empty open subset $\Uu$ of $V$ such that $\dim \Uu _q = \ell(X)$ for all $q\in \Uu$. Thus there is a non-empty open subset $\Uu$ such that for all $q\in \Uu$ either the schemes $\Gamma _q(X)$ are all of {\it type} {\bf A} or of {\it type} {\bf B}.  
\end{remark}

\begin{lemma}\label{typeAgen}
Let $X\subset \PP^r$ be an integral and non-degenerate variety. Let $\Gamma _q(X)\subseteq X$ be the entry locus of $X$ where $q\in \PP^r$ is general.
If $X$ has type {\bf A}, then $r_X(o)=r_{\Gamma _q(X)}(o)$ for a general $o\in \langle \Gamma _q(X)\rangle$ and $\Gamma_o(\Gamma_q(X))= \Gamma _q(X)$ for a general $o\in \langle \Gamma _q(X)\rangle$. 
\end{lemma}

\begin{proof}
Let $\Uu\subset \PP^r$ be a non-empty Zariski open subset of $\PP^r$ from Remark \ref{existenceofopen}, and possibly shrink it so that all $p\in \Uu$ satisfy $r_X(p) = r_{\gen}(X)$. 

We may assume $q\in \Uu$.  Since $\Uu$ is open and $q\in \Uu\cap \langle \Gamma_q(X)\rangle\neq \emptyset$, the latter is an open subset of $\langle \Gamma_q(X)\rangle$. Let $o\in  \Uu\cap \langle \Gamma_q(X)\rangle$. Thus $r_X(o) = r_{\gen}(X)$. Note that the inclusion $\Gamma_q(X)\subseteq X$ implies $r_{\Gamma_q(X)}(o)\geq r_X(o)$. We next show that $r_{\Gamma_q(X)}(o)\leq r_X(o)$. 

Since $X$ is of {\it type} {\bf A}, one has $\Gamma _o(X)=\Gamma _q(X)$. Now, for every $z\in \Gamma_o(X)$, one has $z\in \Gamma_q(X)$. This means that the points in every rank decomposition of $o$ (with respect to $X$) sit inside $\Gamma_q(X)$. Therefore $r_{\Gamma_q(X)}(o)\leq r_X(o)$. Hence $r_{\Gamma_q(X)}(o) = r_X(o)$. 

Since $\Gamma _q(X)\subset X$ and $r_{\Gamma _q(X)}(o)= r_{\gen}(X)$, we have $\Gamma_o(\Gamma _q(X))= \Gamma _q(X)\cap \Gamma _o(X)$. Since $X$ is of type {\bf A}, $\Gamma_o(\Gamma _q(X)) = \Gamma _q(X)$. 
\end{proof}

\begin{remark}
Smooth varieties with quadratic entry locus of \cite{ir,r2} (the so-called $QEL$-manifolds) are of {\it type} {\bf IA}, see \cite[pp. 965--966]{fujrob}: the entry locus
is a smooth irreducible quadric hypersurface in its span. 
\end{remark}

\begin{theorem}\label{x5}
Let $X\subset \PP^r$ be an integral projective variety of dimension $n$. 

\begin{enumerate}
\item[(i)] If $X$ is of {\it type} {\bf A}, then $r_{\gen}(X) = r_{\gen}(\Gamma _q(X))$;

\item[(ii)] Suppose all secants of $\Gamma _q(X)$ have the expected dimension and that $X$ is of {\it type} {\bf I}. Then $X$ has {\it type} {\bf IA} if and only if $\ell (X)=
(\gamma(X)+1)r_{\gen}(X) -\gamma(X)-1$; 

\item[(iii)] Assume $r_{\gen}(X) =2$ and that $X$ is of {\it type} {\bf I}. Then $X$ is of {\it type} {\bf IA} if and only if $\ell(X) = \gamma(X)+1$;

\item[(iv)] Assume $r=2n$ and $X$ smooth. Then $X$ is of {\it type} {\bf IA} if and only if it is a $QEL$-manifold. 
\end{enumerate}
\end{theorem}
\begin{proof}

(i). This is a direct consequence of Lemma \ref{typeAgen}. \\
(ii). Let $\Uu$ be as in the proof of Lemma \ref{typeAgen}. Fix $q \in \Uu$ and assume that $X$ is of {\it type} {\bf IA}. Fix a general $o\in \langle \Gamma _q(X)\rangle$. Since $q\in \Uu$, we may assume $o\in \Uu$ and hence $r_X(o) =r_{\gen}(X)$. Since $X$ is of {\it type} {\bf IA}, $\Gamma _o(X)$ is irreducible and of dimension $\gamma(X)$ and $\Gamma _o(\Gamma _q(X))=\Gamma _q(X)$. Now assume that all secant varieties of $\Gamma _q(X)$ have the expected dimension. 

By (i), $r_{\gen}(\Gamma _q(X)) =r_{\gen}(X)$. Therefore 
\begin{equation}\label{ineqtypeA}
(\gamma(X)+1)(r_{\gen}(X)-1) \le \ell(X)\le (\gamma(X)+1)r_{\gen}(X)-1,
\end{equation}
where the left-hand side is $\dim\sigma_{r_{\gen}-1}(\Gamma_q(X))+1$, the right-hand side is the dimension of $\sigma_{r_{\gen}}(X)$, as $\Gamma_q(X)$ is  non-defective. 

By Lemma \ref{typeAgen}, we have the equality $\Gamma _o(\Gamma _q(X))=\Gamma _q(X)$; so $\Gamma_o(\Gamma_q(X))$ has dimension $\gamma(X)$ for a general  $o\in \langle \Gamma _q(X)\rangle$.  By Lemma \ref{dimentrylocus} applied to the variety $\Gamma_q(X)$, we see that 
\[
\gamma(X) = \dim\sigma_{r_{\gen}}(\Gamma_q(X)) + \dim \Gamma_q(X) +1 -\ell(X),
\]
which implies $\ell(X) =  (\gamma(X)+1)(r_{\gen}(X)-1)$. 

For the converse, suppose $\ell(X) =  (\gamma(X)+1)(r_{\gen}(X)-1)$. Then 
\[
\dim \Gamma_o(\Gamma_q(X))=\gamma(X),
\] 
for a general $o\in \langle \Gamma_q(X)\rangle$. Since $X$ is of {\it type} {\bf I}, $\Gamma_q(X)$ is irreducible of dimension $\gamma(X)$. Thus $ \Gamma_o(\Gamma_q(X))\subset \Gamma_q(X)$ coincide, proving that $X$ is of {\it type} {\bf IA}. \\

(iii). Let $q\in \PP^r$ be general. Suppose $\ell(X) = \gamma(X)+1$. Then $\Gamma_q(X)$ is a hypersurface in its span $\langle \Gamma _q(X)\rangle$. 
Using inequality \eqref{ineqtypeA}, we derive that $X$ is of {\it type} {\bf IA}. 

Suppose $X$ is of {\it type} {\bf IA}. Since $r_{\gen}(X) =2$, a general point $o\in \langle \Gamma_q(X) \rangle$ satisfies $r_{\Gamma _q(X)}(o) =2$. Since $\dim\Gamma_o(\Gamma_q(X))=\gamma(X)$, there is a $\gamma(X)$-dimensional family of sets $S\subset \Gamma _q(X)$ such that $\sharp (S)=2$ and $o\in \langle S\rangle$. Therefore $\ell(X)=\gamma(X)+1$. \\

(iv). Suppose $X$ is of {\it type} {\bf IA}. By Lemma \ref{dimentrylocus}, $\gamma(X) =1$. By (iii), $\ell(X) =2$, i.e. a general $\Gamma _q(X)$ is a plane curve of degree $d>1$. Since $q$ is general in $\PP^r$ and $r_{\gen}(X)=2$, $\Gamma _q(X)$ contains two general points $x, y$ of $X$. By the Trisecant lemma, the set $\{x,y\}$ is the scheme-theoretic intersection between $X$ and the line $\langle \{x,y\}\rangle$. Thus $d=2$. Since $X$ is of {\it type} {\bf I}, $\Gamma_q(X)$ is an irreducible conic. Thus $X$ is a $QEL$-manifold. 

For the converse, suppose $X$ is a $QEL$-manifold. By \cite[pp. 965--966]{fujrob}, $\Gamma _q(X)$ is an irreducible conic, and so $X$ is of {\it type} {\bf I}. 
Therefore $\gamma(X) = 1$ and $\ell(X) = \gamma(X)+1$. By (iii), $X$ is of type {\bf IA}.  
\end{proof}

\begin{example}
If $\gamma(X)=1$ and $X$ is an irreducible curve, then $X$ is of type {\bf I}. Moreover, Theorem \ref{x5}(ii) implies $\ell (X) = 2r_{\gen}(X)-2$ if and only if $X$ has {\it type} {\bf A}. 
\end{example}

\begin{remark}\label{x6}
There are many non-degenerate smooth $X\subset \PP^{2n}$ of dimension $n$ and not of {\it type} {\bf IA}. For instance, one may take the complete intersection of $n$ hypersurfaces of degrees $d_1,\ldots, d_{n}\geq 2$ such that $d_1+\cdots +d_{n}\ge 2n+1$ (so that the canonical bundle has global sections). By Theorem \ref{x5} and \cite[Theorem 2.1(i)]{ir} each such $X$ is either of {\it type} {\bf IB}, {\bf IIB} or {\bf IIA}, as otherwise $X$ would be rational. 

\begin{quote}
\noindent {\it Claim}: $X$ is not of {\it type} {\bf IIA}. \\
\noindent {\it Proof of the Claim}: On the contrary, suppose $X$ is of {\it type} {\bf A}. Fix a general $q\in \PP^{2n}$ and assume $\Gamma _o(X) =\Gamma _q(X)$ for a general $o\in \langle \Gamma _q(X)\rangle$.
Assume for the moment $\ell(X) =2$, i.e. $\Gamma _q(X)$ is a plane curve. As in the proof of Theorem \ref{x5}(iv), we see that $\Gamma _q(X)$ contains a plane conic, 
then $X$ would be a {\it local} $QEL$-manifold (called a $LQEL$-manifold) which is rational \cite[Theorem 2.1(iii)]{ir}, contradicting the assumption on $X$. 

Now assume $\ell(X)>2$. Fix a general $S = \{p_1,p_2\}\in \Ss (X,q)$. Since $X$ is of {\it type} {\bf A}, we see that 
there are infinitely many decompositions $S =\{p_1,p_2\}$ with $p_i\in C_i$, $i=1,2$, for two distinct components $C_1,C_2$ of $\Gamma_q(X)$. 

Since $q\in \langle S\rangle$, $q$ is contained in the join $\JJ(C_1,C_2)$ of the irreducible curves $C_1$ and $C_2$. Thus, for a general $o\in \JJ(C_1,C_2)$, the entry locus $\Gamma _o(X)$ contains the closure of sets $S' =\{z_1,z_2\}$ with $z_i\in C_i$, $i=1,2$. Since any join of irreducible curves has the expected dimension, one has $\dim \JJ(C_1,C_2) =3$, unless $\dim \langle C_1\cup C_2\rangle =2$; the latter case was already considered in the first paragraph of this proof.  Then we are left with the case $\dim \JJ(C_1,C_2) =  3$. A dimension count shows that, for a general $o\in \JJ(C_1,C_2)$, there are only finitely many $S'\subset C_1\cup C_2$ such that $\sharp (S')=2$ and $o\in \langle S'\rangle$. Thus $\Gamma_o(X)\neq \Gamma_q(X)$, which is a contradiction. 
\end{quote}

\end{remark}

Therefore we have the following

\begin{corollary}\label{completeintersections}
Any smooth complete intersection $X\subset \PP^{2n}$ of $n$ hypersurfaces of degrees $d_1,\ldots, d_{n}\geq 2$ such that $d_1+\cdots +d_{n}\ge 2n+1$ 
is of type {\bf B}. 
\end{corollary}

\section{Entry loci of surfaces}\label{surfaces}

To study entry loci for surfaces $X\subset \PP^4$ when $X$ is a linearly normal smooth surface of $\PP^4$, one may employ a result of Franchetta \cite{fra}. Although
we believe the result on entry loci of surfaces featured in Theorem \ref{m1} is probably known to experts, we could not find a reference with an explicit proof. 
We include a proof here for the sake of completeness. First, we recall the following classical definition \cite[Def. 7]{mp}: 

\begin{definition}\label{m2}
Let $Y\subset \PP^3$ be an integral surface. $X$ is said to have {\it ordinary singularities} if either it is smooth or its
singular locus is a reduced curve $D\subset Y$ with the following properties:
\begin{enumerate}
\item[(i)] $D$ is either smooth or it has finitely many ordinary triple points (i.e., points with exactly three branches, all of them
smooth and whose tangent cone spans $\PP^3$);
\item[(ii)] a smooth point of $D$ is either an ordinary non-normal node of $Y$ (i.e., at this point, $Y$ has $2$ smooth branches with distinct
tangent planes) or a pinch point of $Y$ (i.e., at this point, $Y$ is analitically equivalent to $x^2-zy^2=0$ locally at $(0,0,0)\in \CC^3$);
\item[(iii)] $Y$ has only finitely many pinch points;
\item[(iv)] at each triple point of $D$, $Y$ has an ordinary triple point (i.e., at this point, $Y$ is analitically equivalent to the surface $xyz=0$ locally at
$(0,0,0)\in \CC^3$).
\end{enumerate}
\end{definition}

\begin{theorem}\label{m1}
Let $X\subset \PP^4$ be a smooth and non-degenerate surface. A general $\Gamma _q(X)$ is reducible if and only if $X$ is an isomorphic linear
projection of the Veronese surface.
\end{theorem}
\begin{proof}
Fix a general $q\in \PP^4$ and set $Y:= \pi _q(Y)$ and $\pi:= \pi _{q|X}: X\to Y$. By
\cite[Theorem 8]{mp}
$Y$ is an integral surface of degree
$\deg (X)$ with ordinary singularities. Set $C:= \pi^{-1}(D)$. We have $\Gamma _q(X) = C$.

First assume  that $X$ is an isomorphic projection of the
Veronese surface. In this case $Y$ is the {\it Steiner's Roman Surface}, i.e., a degree $4$ integral surface with a triple point,
$p$, and $D$ is the union of $3$ lines through $p$. The curve $C$ is a union of $3$ smooth conics, each of them sent by the $2:1$ map $\pi$ 
onto a line of $D$.

Now assume that $X$ is not an isomorphic linear projection of the Veronese surface. We show that $C$ is irreducible. 
By a theorem of Franchetta (\cite{fra}, or see \cite[Theorem 5]{mp} for a complete proof
in modern language) $D$ is irreducible. ($C$ and $D$ are Weil divisors and their interpretation in terms of conductors is
explained in \cite[\S 2]{pi}.) Assume that
$C$ is not irreducible. Let $C_i$ be the irreducible components of $C$. Since $\pi_{|C}$ has degree two, one has
$2 = \deg(\pi_{|C}) = \sum_{i=1}^t \deg(\pi_{|C_i})$; thus $C$ has two irreducible components, $C_1$ and $C_2$. 
Now, the map $\pi_{|C_i}$ is an isomorphism over each smooth point of $D$, because it is bijective. 

Call $m\in \NN$ the number of pinch points of $Y$. Since $Y$ have only ordinary singularities, its pinch points are smooth points of $D$ and they correspond to tangent lines $L$ to $X$ passing through $q$ with $\deg (L\cap X)=2$. Since $\tau (X)=\PP^4$  (see e.g. \cite[Th. 5.1 or Example after Corollary 5.2]{fl}, \cite[Theorem 1.4]{z}), we have $m>0$. Fix a pinch point $p\in D$ and let $z$ be the corresponding point of $C$. Since $z$ is the only point of $X$ with $\pi (z) = p$, one derives $z\in C_1\cap C_2$. Thus $z$ is a singular point of $C$, contradicting smoothness of $C$ proven in \cite[p. 618]{GH}: they give explicitly local 
equations for $D$ and its cover $C$ (in their notation, $C$ corresponds to $s=0$ locally at the origin), showing that they are smooth at our points $p$ and $z$ respectively. 
\end{proof}

\begin{remark}
Let $X\subset \PP^4$ be a degree $d$ integral and non-degenerate surface with only isolated singularities. Note that we may take as $X$ a general linear projection of a smooth and non-degenerate surface $X'\subset \PP^r$, $r\ge 5$, i.e. we are not assuming that the singular points of $X$ (if any) are normal singularities. 
Let $g$ be the arithmetic genus of a general hyperplane section $X\cap H$ of $X$, where $H\subset \PP^4$ is a general hyperplane. Since $X$ has only finitely many singular points, Bertini's theorem gives that $X\cap H$ is a smooth and irreducible curve \cite[II.8.18 and III.7.9]{h}. Thus $X\cap H$  has geometric genus $g$. Since $X$ is non-degenerate, $X\cap H$ spans $H$ and hence we may apply Castelnuovo's upper bound for $g$ in terms of $d$ \cite[Theorems 3. 7 and 3.13]{he}. When $X$ is a birational linear projection of a smooth surface $X'\subset \PP^r$ ($r\ge 5$), $d$ and $g$ are then the invariants of a smooth and non-degenerate curve of $\PP^{r-1}$. Castelnuovo's  upper bound for the genus of curves in $\PP^{r-1}$ excludes more pairs $(d,g)$; see \cite[3.7, 3.11, 3.15, 3.17, 3.22]{he}.
\end{remark}

\begin{remark}
Let $X\subset \PP^4$ be an integral and non-degenerate surface with only isolated singularities. Let $\Uu$ be a non-empty open subset of $\PP^4$ as in Remark \ref{existenceofopen}. The entry loci give rise to a natural invariant of $X$ as follows. 

We have an algebraic family $\{\Gamma _q(X)\}_{q\in \Uu}$ of effective Weil divisors of $X$, whose irreducible components are reduced.  
Let $\Vv$ be the image of the family $\{\Gamma _q(X)\}_{q\in \Uu}$
in the Hilbert scheme of $X$. The family $\Vv$ may have lower dimension, as for a general $q\in \Uu$ there may exist infinitely many $p\in \Uu$ with $\Gamma _p(X)=\Gamma _q(X)$. Its dimension $\dim \Vv$ is another invariant of the embedded variety $X$. 
\end{remark}

To study entry loci for surfaces in $\PP^4$, it is useful to first record the following 

\begin{proposition}\label{degentryloc}
Let $X\subset \PP^4$ be a degree $d$ integral and non-degenerate surface of sectional genus $g$ and let $q\in \PP^4$ be general. 
Then $\Gamma_q(X)$ is a curve of degree  $\deg (\Gamma _q(X)) = (d-1)(d-2)-2g$. 
\begin{proof}
By Lemma \ref{dimentrylocus}, $\Gamma_q(X)$ is a curve. Let $H\subset \PP^4$ be a general hyperplane containing $q$. Since $q$ is general, 
the hyperplane section $C = X\cap H$ is a genus $g$ smooth, connected and non-degenerate
curve of $H$. Let $\Uu$ be a non-empty open subset of $\PP^4$ as in Remark \ref{existenceofopen}. We may assume $q\in \Uu$. Let $\pi_q$ be the linear projection from $q$; thus $\pi_q(C)$ is a curve with ordinary nodes of degree $d$ and geometric genus $g$. Hence $\pi_q(C)$ has $(d-1)(d-2)/2 - g$ nodes. The set of nodes of $\pi_q(C)$ is in bijection with the elements of $\Ss(C,q)$. So $\sharp (\Ss (C,q)) =(d-1)(d-2)/2-g$ and hence, by definition, $\sharp (\Gamma _q(X)\cap H) =(d-1)(d-2)-2g$. Therefore $\deg (\Gamma _q(X)) = (d-1)(d-2)-2g$. 
\end{proof}
\end{proposition}

\begin{example}
A complete intersection $X\subset \PP^4$ of two general hypersurfaces of degrees $d_1=2$ and $d_2=3$ is a smooth canonically embedded K3 surface of degree $6$ and genus $4$. By Corollary \ref{completeintersections}, $X$ is of {\it type} {\bf B}. By Theorem \ref{m1} and Lemma \ref{degentryloc}, 
the entry locus $\Gamma_q(X)$ is an irreducible curve of degree $12$. 
Let $D$ be the reduced singular locus of the projection $\pi_q(X)$ of $X$ from the point $q$; $D$ is a curve of degree $6$ and arithmetic genus $4$ \cite[Example 4.17]{ma} and it is the base locus of the inverse of the birational morphism $\pi_{q_{|X}}: X\rightarrow \pi_q(X)$. The entry locus $\Gamma_q(X)$ is a double cover of $D$. 
\end{example}

\begin{lemma}\label{uuuu0}
Let $X\subset \PP^3$ be an integral and non-degenerate curve. 
\begin{enumerate}
\item[(i)] If $X$ is not a rational normal curve then $\sharp (\mathfrak{S}(\Gamma_q(X)))=1$.  

\item[(ii)] If $X$ is a rational normal curve then $\Sigma(\Gamma_q(X))$ is a line with two points removed.

\end{enumerate}
\end{lemma}

\begin{proof}
Fix a general $q\in \PP^3$. We have $r_X(q) =2$. Moreover, $\sharp (\Ss (X,q)) =1$ if and only if $X$ is a rational normal curve \cite[Theorem 3.1]{cc2}.\\
(i). Assume $X$ is not a rational normal curve and take $A, B\in \Ss (X,q)$ such that $A\ne B$. The sets $\langle A\rangle$ and $\langle B\rangle$
are lines containing $q$. Thus either $\langle A\rangle = \langle B\rangle$ or $\langle A\rangle \cap \langle B\rangle=\emptyset$. In the latter case, $q$
is unique point contained in all $\langle D\rangle$, $D\in \Ss (X,q)$.\\
(ii). Assume $X$ is a rational normal curve and write $\{A\} = \Ss (X,q)$. All points of $q'\in \langle A\rangle \setminus A$ have rank two and such that
$\Gamma_q'(X) = A$, because $X$ has no trisecant lines. Therefore $A =\Gamma _{q'}(X)$ if and only if $q'\in \langle A\rangle \setminus A$. 
\end{proof}

\begin{corollary}\label{uuuu1}
Let $X\subset \PP^4$ be an integral, normal and non-degenerate surface. Assume that $X$ is a cone with vertex $v$. Fix a general $q\in \PP^4$. For any hyperplane $H\subset \PP^4$ such that $v\notin H$, set $C_H:= H\cap X$ and $\{q_H\}:= \langle \{v,q\}\rangle\cap H$.
\begin{enumerate}

\item[(i)] All triples $(H,C_H,q_H)$ are projectively equivalent and $r_{C_H}(q_H) =2$ for all $H\subset \PP^4\setminus \{q\}$.

\item[(ii)] Fix $H\subset \PP^4\setminus \{v\}$. Then $\Gamma _q(X)$ is the union of $2\times \left(\sharp (\Ss (C_H,q_H)\right)$ lines through $v$ and it is the cone with vertex $v$ and a set $A\in \Ss(C_H,q_H)$ (for fixed $H$) as base.

\item[(iii)] Assume $X$ normal. Set $d:= \deg (X)$ and $g:= p_a(C_H)$. Then $\Gamma _q(X)$ is the union of $(d-1)(d-2)-2g$
lines through $v$. So $X$ has {\it type} {\bf II}. 

\item[(iv)] If $\deg (X) \ne 3$ then $\Sigma(\Gamma_q(X)) =\langle \{v\}\cup q\rangle \setminus \{v\}$ for a general $q\in \PP^4$.

\item[(v)] If $\deg (X)=3$ then $\Sigma(\Gamma_q(X))$ is a plane through $v$ with two lines through $v$ removed.
\end{enumerate}
\end{corollary}

\begin{proof}
(i). Fix hyperplanes $H, M\subset \PP^4\setminus \{v\}$. Let $\pi : \PP^4\setminus \{v\} \to \PP^3$ be the projection from $v$. Note that $\pi _{|H}: H\to \PP^3$ and $\pi _{|M}: M\to \PP^3$ are isomorphisms.

The automorphism $f = \pi _{|M}^{-1}\circ \pi _{|H}: H\to M$ satisfies $f(q_H)=f(q_M)$. Since $X$ is a cone with vertex $v$, $f$ induces a linear isomorphism $f_{|C_H}: C_H\to C_M$. Thus $r_{C_H}(q_H) = r_{C_M}(q_M)=2$. Since $q$ is general in $\PP^4$, $q_H$ is general in $H$. Thus $r_{C_H}(q_H)=2$ and the integer $\sharp (\Ss (C_H,q_H))$ is the same for general $q$ and general hyperplanes.

(ii). As in the proof of Proposition \ref{degentryloc}, one has 
\[
\sharp (\Ss (C_H,q_H)) = (d-1)(d-2)/2-g. 
\]

For a general $q_H\in H$, as $q$ is outside the locus of trisecants to $C_H$, we also have $A\cap B =\emptyset$ for all $A, B\in \Ss (C_H,q_H)$ such that $A\ne B$. Thus, varying the hyperplane $H\subset \PP^4$ such that $v\notin H$, we see that $\Ss (X,q)\supseteq \cup_{A\in \Ss(C_H,q_H)}A$; the latter is $D\setminus \{v\}$, where $D$ is the cone with vertex $v$ and base given by a (arbitrary) finite set $A\in \Ss(C_H,q_H)$ (for fixed $H$). Hence $\Gamma _q(X)\supseteq D$. 

To see that $D = \Gamma _q(X)$, let $z\in \Gamma_q(X)\setminus D$. Then there exists $p\in X$ such that $q\in \langle z,p\rangle$. Since $z\neq v$, 
there exists a hyperplane $H\subset \PP^4$ such that $v\notin H$ and $\langle z,p\rangle \subset H$; note that, using the notation above, $q=q_H$. 
Then $\lbrace z,p\rbrace \in \Ss(C_H, q_H)$ and so $\langle z,p\rangle\subset D$, a contradiction. 

(iii). This is a consequence of the description of $\Gamma_q(X)$ given in (ii). 

(iv). Fix a hyperplane $H\subset \PP^4$ such that $v\notin H$. This statement is  equivalent to the equality $\mathfrak{S}(\Gamma_{q_H}(C_H)) = \{q_H\}$, as in Lemma \ref{uuuu0}(i).

(v). This is a consequence of Lemma \ref{uuuu0}(ii). 
\end{proof}

\begin{proposition}\label{mindegsurf}
Let $X\subset \PP^4$ be an integral and non-degenerate surface with only isolated singularities of minimal degree $\deg(X)=3$. 
Then $X$ is of {\it type} {\bf A}. Moreover, $X$ is of {\it type} {\bf II} if and only if it is a cone. 
\begin{proof}

Such surfaces are classically classified, and their general hyperplane section $X\cap H$ is a rational normal curve: either $X$ is a cone over a rational normal curve of $\PP^3$ or it is isomorphic to the rational ruled surface $\mathbb F_1$ embedded by the complete linear system $|h+2f|$. In both cases, by Proposition \ref{degentryloc}, $\Gamma _q(X)$ is a conic. Thus $\gamma(X)=1$ and $\ell(X) = \gamma(X)+1$. In this case, one can slightly improve Theorem \ref{x5}(iii): for any $o\in \langle \Gamma_q(X)\rangle$, one can check that $\Gamma_o(X)$ coincides with the conic $\Gamma_q(X)$, and so $X$ is of {\it type} {\bf A}. 
This shows the first statement. 

Suppose $X$ is a cone with vertex $v$. Then apply Corollary \ref{uuuu1}(iii). 

Now assume $X =\mathbb F_1$. The lines of $X$ are only the section $h$ and the fibers $|f|$ of the ruling. Since $X$ has only $\infty ^1$-many lines, a general $\Gamma _q(X)$ must be irreducible and so $X$ is of {\it type} {\bf I}. This shows the statement. 
\end{proof}
\end{proposition}

\begin{theorem}\label{surfdegfour}
Let $X\subset \PP^4$ be an integral and non-degenerate surface with only isolated singularities of almost minimal degree $\deg(X)=4$ and sectional
genus $g=1$. 
\begin{enumerate}
\item[(i)] When $X$ is a cone with vertex $v$, $\Gamma _q(X)$ is a hyperplane section of $X$ containing $v$, i.e. it is formed by four lines through $v$, no three of them coplanar.
\item[(ii)] When $X$ is not a cone, $\Gamma_q(X)$ coincides with the general hyperplane section and $\sharp \mathfrak{S}(\Gamma _q(X)))=4$. 
\end{enumerate}

\begin{proof}

The general hyperplane section $C =X\cap H$ is a quartic curve by Lemma \ref{degentryloc}. The classification of surfaces in the statement is given in \cite{bp} and \cite[8.11 and 8.12]{fuji}. The surface $X$ is the complete intersection of two quadrics;  moreover, $C$ is a complete intersection of two quadric surfaces and hence it is a linearly normal elliptic curve. 

Assume $X$ is a cone and let $v$ be its vertex. Take a hyperplane $H\subset \PP^4$ such that $v\notin H$ and let $q' = \langle \{v,q\}\rangle \cap H$.
Since $X\cap H$ is a smooth space curve of degree three and $q$ is general, $q'$ is a general point in $H$. Thus $\sharp (\Ss (C,q')) =2$, i.e. $\Gamma _q(X)\cap H$ is formed by four points (as $C$ is a linearly normal elliptic curve), and $\Gamma _q(X)$ is the union of the four lines passing through $v$ and containing one of these points.

Assume $X$ is not a cone. In this case, $X$ does not
contain $\infty ^1$-many lines. Indeed, $X$ is a complete intersection of quadrics in $\PP^4$ and so $\omega_X \cong \Oo_X(-1)$. 

Assuming there are infinitely many lines, since there is a finite number of singular points of $X$, there exists a line $L\subset X_{\reg}$. 
Since there is a positive dimensional family of lines,  $L^2\geq 0$, i.e. the degree of the normal bundle of $L$ is non-negative. Note that $\omega_X\cdot \Oo_X(L) = -\deg(L) = -1$; 
however, this contradicts adjunction: $-2 = \deg(\omega_L) = L^2 + \omega_X\cdot L$. Thus there are only finitely many lines. This implies that the entry locus $\Gamma _q(X)$ is either irreducible or a union of two irreducible conics. 

Suppose $\Gamma_q(X) = D\cup D'$, where $D, D'$ are conics (possibly the same). Since $X$ is a surface, 
there is no a two-dimensional family of conics. Therefore, there are at most a finite number of families
of conics, each being at most one-dimensional. Suppose $C$ and $C'$ belong to two (possibly the same) families $\mathcal F, \mathcal F'$, 
each of dimension at most one. 

We have two cases according to the value of $\dim \langle D\cup D'\rangle$. 
Suppose $\langle D\cup D'\rangle = \PP^4$. Then Lemma \ref{y9} implies that there is no $q\in \PP^4$ such $\Gamma_q(X)= D\cup D'$. 
Suppose $\langle D\cup D'\rangle = \PP^3$. By Lemma \ref{y5} there are at most $\infty^1$-many points in $\PP^4$ for fixed conics $D, D'$. 
However, since $\mathcal F$ and $\mathcal F'$ are at most one-dimensional, the points $q\in \PP^4$ such that $\Gamma_q(X)= D\cup D'$ are $\infty^3$-many;
this is a contradiction as $q$ was general in $\PP^4$. 

We show that for a general hyperplane section $C:= X\cap H$ of $X$ there are exactly four points $o_1,o_2,o_3,o_4\in H$ such that $C = \Gamma _{o_i}(X)$. 
Since $C$ is a linearly normal elliptic curve, the following facts are well-known (see e.g. Case (2) in the proof of \cite[Theorem 1]{piene}):
\begin{enumerate}
\item[(i)] $C$ is contained in exactly $4$ quadric cones and their vertices $o_1,o_2,o_3,o_4$ are distinct;
\item[(ii)] $r_C(o_i) =2$ for all $i$;
\item[(iii)] $o_1,o_2,o_3,o_4$ are the only points inside $p\in H$ such that $r_C(p)=2$ {\it and} such that $\Ss (C,p)$ is infinite. These are equivalent to $r_C(p)=2$ and $\Gamma _p({C}) = C$; the latter equality holds because $C$ is an irreducible curve. 
\end{enumerate}

Since $o_i\in H$ and $o_i\notin C = X\cap H$, we have $r_X(o_i) > 1$. Since $r_C(o_i) =2$ and $\Gamma _{o_i}({C}) =C$,
we get $r_X(o_i) =2$ and $\Gamma _{o_i}(X)\supseteq C$. Since $\Gamma _{o_i}(X)$ is a degree $4$ curve, one has $\Gamma _{o_i}(X)=C$. 

For a general $q\in \PP^4$, we showed above that $\Gamma_q(X)$ is an irreducible curve of degree four. Note that $\Gamma_q(X)$ cannot be
a rational normal curve because for every hyperplane $H$ containing $q$, $q$ is spanned by four points on $C = X\cap H$ which are coplanar. 
The curve $\Gamma_q(X)$ cannot be a plane curve of degree $d=4$ either, by the very same argument in the proof of Theorem \ref{x5}(iv). Thus $\langle \Gamma_q(X)\rangle$ is a hyperplane, for $q\in \Uu$, where the latter is an open set in $\PP^4$ where $\Gamma_q(X)$ have the same degree and arithmetic genus for each $q$ as in Remark \ref{existenceofopen}. Define the following
morphism: 
\[
\psi: \Uu \longrightarrow \PP^{4\vee},
\]
\[
q\longmapsto H_q:=\langle \Gamma_q(X)\rangle. 
\]
Since for any integral curve $Y\subset \PP^3$, its Segre set $\Sigma(Y)$ is finite, it follows that this map is finite. Thus it is dominant by dimensional count. 
Hence the general hyperplane section is of the form $C=X\cap H_q = \Gamma_q(X)$. Moreover, property (iii) of the linearly normal elliptic curve $C$ stated above yields $\deg(\psi) = 4$ and so $\sharp(\Sigma(\Gamma_q(X))) = 4$.  
\end{proof}
\end{theorem}

\section{A large class of varieties}\label{afamily}

Let $n\ge 2$ and $s\ge 3$. We introduce the following class of projective varieties: 

\begin{definition}\label{largeclass}
Let $r = s(n+1)-2$. Let $\Ff (n,s)$ be the set of all integral and non-degenerate $n$-dimensional varieties $X\subset \PP^r$ such that $\sigma _s(X) =\PP^r$ and $X$ has {\it type} {\bf I}. Note that $r_{\gen}(X) = s$. By assumptions on $r$ and $X$, Lemma \ref{dimentrylocus} implies that $\Gamma _q(X)$ is an irreducible curve.  
\end{definition}

Let $X\in \Ff (n,s)$. Since $\Gamma_q(X)$ is non-defective, $2s-2 \le \ell (X)\le 2s-1$. Theorem \ref{x5}(ii) shows that $\ell (X)=2s-2$ if and only if $X$ is of {\it type} {\bf IA}. 

The next results explain why $\Ff (n,s)$ is a large class of varieties. We start with an auxiliary lemma: 

\begin{lemma}\label{uuu2}
Fix a an integral projective variety $W$ of dimension $n\geq 2$, and a very ample line bundle $L$ on it. 
Assume $h^1(L^{\otimes t})=0$ for all $t\ge 1$. Fix integers $k$ and $m$ such that $k\ge 2m >0$. 
Let $j: W\to \PP^N$ be the embedding defined by the linear system $|L^{\otimes k}|$ and set $Y=j(W)$. Then $\dim \sigma _m(Y) = m(n+1)-1$ and $\sharp (\Ss (Y,q)) =1$ for all $q\in \PP^r$ such that $r_X(q)\le m$. \end{lemma}

\begin{proof}
We first prove the following claim.

\begin{quote}
\indent {\it Claim :} For each integer $i\ge 1$ and each $S\subset W$ such that $\sharp (S) =i$ we have $h^1(\Ii _S\otimes L^{\otimes i}) =0$ and $h^0(\Ii _S\otimes L^{\otimes i})= h^0(L^{\otimes i})-i$. \\

\noindent {\it Proof of the Claim:} If $i =1$, the statement holds, because $L$ is globally generated (equivalently, $L$ does not have
base points) and $h^1(L)=0$. Similarly, note that, for every $i\geq 2$, $L^{\otimes i}$ is globally generated and $h^1(L^{\otimes i})=0$ implies $h^1(\Ii _p\otimes L^{i})=0$,
for every $p\in W$.  

Assume $i\ge 2$. Let $S\subseteq W$ such that $\sharp (S) =i$. Fix $p\in S$ and set $A = S\setminus \lbrace p\rbrace$. Since $L$ is very ample and $A$ is finite, there exists a section $H\in |L|$ such that $H(p) = 0$ and {\it not} vanishing on any of the points in $A$. Let $D$ be the support of $H$. Thus we have the exact sequence: 
\begin{equation}\label{equuu2}
0 \to \Ii _A\otimes L^{\otimes (i-1)} \to \Ii _S\otimes L^{\otimes i} \to \Ii _{p,D}\otimes (L_{|D})^{\otimes i} \to 0,
\end{equation}
where $L_{|D}$ is $L$ restricted on $D$. By induction, $h^1(\Ii _A\otimes L^{\otimes (i-1)})=0$. Since $L_{|D}^{\otimes i}$ is very ample, one has $h^1(\Ii _{p,D}\otimes (L_{|D})^{\otimes i})=0$ as before. This proves the {\it Claim}.  
\end{quote}

Since $k\ge 2m$ and $h^1(L^{\otimes k})=0$, the {\it Claim} implies $h^1(\Ii _S\otimes L^{\otimes k})=0$ for each set $S\subset W$ such that $\sharp (S)\le 2m$,
i.e. any subset of $Y$ with cardinality at most $2m$ is linearly independent in $\PP^N$. Thus $\sharp (\Ss (Y,q)) =1$ for all $q\in \PP^r$ such that $r_X(q)\le m$.
This implies that the first $m$ secant varieties of $Y$ have expected dimension $\dim \sigma _m(Y) =m(n+1)-1$. \end{proof}

We are now ready to prove the following: 

\begin{theorem}\label{u1}
Fix an integer $s\ge 3$. Let $W$ be a smooth and irreducible projective variety. Then there exists $X\in \Ff (n,s)$ such that $X\cong W$.
\end{theorem}

\begin{proof}
Fix a very ample line bundle $L$ on $W$ such that $h^i(L^{\otimes t})=0$ for all $i\geq 1$ and all $t\geq 1$. Let $j: W\to \PP^N$ be the embedding induced by the complete linear system $|L^{\otimes k}|$ for some $k\ge 2s$. Let $Y = j(W)$, and $r = s(n+1)-2$. 

By Lemma \ref{uuu2}, the first $s$ secants of $Y$ have the expected dimension. Note that $\dim \sigma_s(Y) = s(n+1) -1 = r+1$. 

Lemma \ref{uuu2} shows that for every $p\in \PP^N$ such that $r_Y(p) \leq s $, one has $\sharp (\Ss (Y,p))=1$. This implies $N\geq \dim \sigma_s(Y) = r+1$. 
Indeed, otherwise, $\sigma_s(Y)$ would have already filled up the ambient space $\PP^N$ with the constructible set $\Ss (Y,q)$ being positive dimensional for a general $q\in \PP^N$, which contradicts Lemma \ref{uuu2}. 

Let $V\subset \PP^N$ be a general linear subspace with $\dim V = N-r-1$ and let $X = \pi_V(Y)\subset \PP^r$ be the projection of $Y$ from $V$. The variety $X$ is isomorphic to $Y$, as $V\cap \sigma_2(Y)=\emptyset$, by the generality of $V$. Hence $X\cong W$. Notice that the general point $p\in \PP^r$ is the projection of a general point $q\in \sigma_s(Y)$. Indeed, the projection $\pi_V$ induces a rational dominant map from $\sigma_s(Y)$ to $\PP^r$.

Now we prove that $X$ has {\it type} {\bf I}. Note that $r_{\gen}(X) = s$. By Lemma \ref{dimentrylocus}, $\gamma(X) = \dim \Gamma _q(X)=s(n+1)-r-1 = 1$. Fix a general $q\in \PP^r$. So there is $p\in \sigma _s(Y)$ such that $q =\pi _V(p)$. Since $q$ is general in $\PP^r$, $p$ is general in $\sigma _s(Y)$ and in particular $r_Y(p)=s$. 
We have $\sharp (\Ss (Y,p)) =1$, say $\Ss (Y,p) = \{S_{p}\}$ with $\sharp (S_{p}) =s$. Let $U = \langle \{p\}\cup V\rangle$. Since $V$ and $p$ are general, $U$ is a general linear subspace of $\PP^N$ of dimension $N-r$. Since $U$ is general, $D = U\cap \sigma _s(Y)$ is an integral curve and $U\cap \sigma _{s-1}(Y) =\emptyset$. The general 
point $z\in D$ satisfies $r_Y(z)=s$ and $\sharp (\Ss (Y,z)) =1$. Denote $D^{\circ}$ an open set of points $z\in D$ with the above properties. Therefore, the set $\Gamma _q(X)$ is the closure of the projection of the constructible curve $\{\pi_V(S_z)\}_{z\in D^{\circ}}$, which is irreducible. 
 \end{proof}

\begin{definition}
Let $\Ff '(n,s)$ (resp. $\Ff''(n,s)$) denote the set of all $X\in \Ff (n,s)$ of {\it type} {\bf IA} (resp. {\it type} {\bf IB}).
\end{definition} 
 
\begin{remark}
Fix $X\in \Ff '(n,s)$ and let $p_g(X)$ (resp. $p_a(X)$, $d_X$) be the geometric genus (resp. arithmetic genus, degree) of the entry locus $\Gamma _q(X)$. 

Let $X\in \Ff (n,s)$ and $\Gamma _q(X)$ be the corresponding general entry locus. The latter passes through $s$ general points of $X$. 
Thus if $X\in \Ff '(n,s)$ there are strong lower bounds for $p_g(X)$ (e.g. $X$ may contain no curve with small geometric genus) and strong lower bounds on $d_X$ (e.g. when $X$ is embedded by a very positive line bundle it does not contain low degree curves).
 \end{remark}

We conclude this section with two natural questions: 

\begin{question}\label{q2}
Fix $X\in \Ff (n,s)$. Is $X\in \Ff '(n,s)$? Assume $X\in \Ff '(n,s)$ and smooth. Is a general entry locus of $X$ a smooth curve?
\end{question}

\begin{question}\label{q3}
What is the minimum of all $p_g(X)$ (resp. $p_a(X)$, $d_X$) among all $X\in \Ff '(n,s)$ and among all smooth $X\in \Ff '(n,s)$? 
\end{question}

\section{Tangentially degenerate curves}\label{tangdeg}

Let $r\ge 3$, and $X\subset \PP^r$ be an integral and non-degenerate curve. We say that $X$ is {\it tangentially degenerate} if a general tangent line $L= T_pX$ of $X$, where $p\in X_{\reg}$, meets $X\setminus \{p\}$. In positive characteristic, many tangentially degenerate smooth curves are known \cite{k1, k2}. In characteristic zero, no tangentially degenerate curve is known. In characteristic zero, neither smooth nor nodal curves are tangentially degenerate \cite[Theorem 3.1 and Remark 3.8]{k1}; moreover, other important classes of singular curves are not tangentially degenerate \cite{bp}. Let $(\dagger)$ denote the following condition: 
\vspace{0.3cm}
\begin{center}
$(\dagger)$ {\it In characteristic zero, there is no tangentially degenerate curve}. 
\end{center}

\begin{remark}
Condition $(\dagger)$ is \cite[Conjecture 1.6]{bp}. 
\end{remark}

\begin{theorem}\label{v1}
Assume $(\dagger)$. Fix integers $n\ge 2$ and $s\ge 3$. Let $r = s(n+1)-2$. Let $X\subset \PP^r$ be an integral and non-defective $n$-dimensional variety. Assume $\sigma _s(X) =\PP^r$
and that $X$ is of {\it type} {\bf IA}. Then $\deg (\Gamma _q(X)) = 2s-2$, $\ell (X) =2s-2$ and $\Gamma _q(X)$ is a rational normal curve of $\langle \Gamma _q(X)\rangle$.
\end{theorem}

Let $X$ be as in the assumption of Theorem \ref{v1}. Assuming $(\dagger)$, this result implies that $X$ is rationally connected and for any $s$ general points of $X$ there passes a rational curve of degree $2s-2$.

\begin{lemma}\label{v0}
Assume $(\dagger)$. Let $X\subset \PP^r$, $r\ge 4$, be an integral and non-degenerate curve. Fix a general tangent line $L$ of $X$  and let $\pi _L: \PP^r\setminus L\to \PP^{r-2}$ be the linear
projection from $L$. Let $f: X\to \PP^{r-2}$ be the morphism induced by $\pi _{L|X\setminus X\cap L}$. Then $f$ is birational onto its image and $\deg (f(X)) =\deg (X)-2$.
\end{lemma}

\begin{proof}
Since $L$ is a general tangent line of $X$, $L$ does not contain singular points of $X$, it is tangent to $X$ at a unique point $p$ with order of contact two, $L= T_pX$.

Since $X\cap L\subset X_{\reg}$ and $\PP^{r-2}$ is a projective variety, $\pi _{L|X\setminus X\cap L}$ extends to a unique morphism $f: X\to \PP^{r-2}$. Assume that $f$ is not birational onto its image, i.e. assume that $L$ meets infinitely many secant lines of $X$ at points outside $X$. 

There are finitely many points $q\in \PP^r\setminus X$ such that $q\in \Sigma(X)$, the Segre set of $X$. Since $L$ is a general tangent, we may assume that $L$ contains no such a point. Thus a general $o\in L$ is contained in at least one line $M$ such that $M$ is not tangent to $X_{\reg}$ and $\sharp (M\cap X_{\reg})\ge 2$. 

Now we prove that for a general $L$ and a general $o\in L$ there is a line $M$, $M$ not tangent to $X_{\reg}$, such that $\sharp (M\cap X_{\reg})= 2$. This is clear if $X_{\reg}$ has only finitely many multisecant lines. Then we may assume that $X_{\reg}$ has infinitely many multisecant lines. By the Trisecant lemma \cite[p. 109]{acgh} the multisecant lines of $X$ are $\infty ^1$-many, i.e. the set $\Phi$ of all multisecant lines is a finite union of irreducible families, say $E_1,\dots ,E_h$, of dimension at most one. For $1\le i\le h$, call $T_i$ the closure in $\PP^r$ of the set $\cup _{R\in E_i} R$. Each $T_i$ is an integral surface ruled by lines. Since we are in characteristic zero, only finitely many lines of each of the rulings of $T_i$ are tangent to $X_{\reg}$. Then a general tangent line to $X_{\reg}$ is not contained in $T_1\cup \cdots \cup T_h$.

In conclusion, $L$ meets only finitely many multisecant lines. So for a general $o\in L$, we have $\deg (M\cap X)=2$ and $M\cap X$ is reduced; let $M\cap X = \{p_1(o),p_2(o)\}$. Since $o$ is general in $L=T_pX$, $p_1(o)$ is a general point of $X$. Let $\psi : \PP^r\setminus  \{p_1(o)\}\to \PP^{r-1}$ be the linear projection from $p_1(o)$. Since $p_1(o)$ is a smooth point of $X$, $\psi _{|X\setminus \{p_1(o)\}}$ extends to a unique morphism $h: X\to \PP^{r-1}$. Set $Y = h(X)$.
\begin{quote}
\indent \emph{Claim :} The morphism $h: X \to \PP^{r-1}$ is birational onto its image.\\

\noindent \emph{Proof of the Claim:} Since we are in characteristic zero, it is sufficient to prove that $h$ is generically injective. Fix a general $z\in X$. Then $z\notin T_{p_1(o)}X$, i.e. $h(p_1(o)) \ne h(z)$. By the Trisecant lemma \cite[p. 109]{acgh}, the line $\langle \{p_1(o),z\}\rangle$ contains no other point of $X$. Thus $h(z)\ne h(b)$ for all $b\in X\setminus \{p_1(o), "z\}$. 
\end{quote}

Since $p_1(o)\in X_{\reg}$,  the {\it Claim} implies $\deg (Y) =\deg (X)-1$. Since $X$ is non-degenerate, $Y$ is non-degenerate. Since $r-1\ge 3$ and we assume $(\dagger)$, the curve $Y'$ is not tangentially degenerate. Since $p$ is general in $X$ and $o$ is general in $L=T_pX$, the pair $(p,p_1(o))$ is general in $X^{\times 2}$. 
Hence $h({p})$ is a general point of $Y$. Note that $p_1(o)\notin L=T_pX$. By assumption, the tangent line $\psi (L) = h(L)$ of $Y$ at $h({p})$ contains the point $h(p_2(o))$ of $Y$. Thus $Y$ is tangentially degenerate, which contradicts $(\dagger)$. 

In conclusion, the map in the statement $f: X\to \PP^{r-2}$ has to be birational onto its image. Moreover, since $\deg (L\cap X)=2$, $\deg (f(X))=\deg (X)-2$.\end{proof}

\begin{proof}[Proof of Theorem \ref{v1}:]
Take a general $q\in \PP^r$. By assumption $r_X(q)=s$ and $\Gamma _q(X)$ is an irreducible curve. Note that Theorem \ref{x5}(ii) implies $\ell(X)=\dim \langle \Gamma _q(X)\rangle = 2s-2 \ge 4$. Thus $\deg (\Gamma _q(X))\ge 2s-2$ and equality holds if and only if $\Gamma _q(X)$ is a rational normal curve of $\langle \Gamma _q(X)\rangle$. 

Since $\Ss (X,q)\subseteq \Gamma _q(X)$ and $q$ is general, we may assume that $\Gamma _q(X)$ contains a general $S\subset X$ such that $\sharp (S)= s$.
Such an $S$ is also a general subset of $\Gamma _q(X)$ with cardinality $s$. Thus $S\subset X_{\reg}\cap \Gamma _q(X)_{\reg}$. Fix $A\subset S$ such that $\sharp (A)=s-2$, so $A\subset X_{\reg}$. 

Now, let 
\[V = \langle \cup _{p\in A} T_pX\rangle \mbox{ and } W = \langle \cup _{p\in A} T_p\Gamma _q(X)\rangle.
\]

Since $\Gamma _q(X)$ contains $s$ general points of $X$ and $X$ is non-defective, Terracini's lemma \cite[Corollary 1.4.2]{r3} gives $\dim V=(s-2)(n+1) -1$ and $\dim W = 2s-5$. Let $X' \subset \PP^{2n}$ be the closure of the image of $X\setminus X\cap V$ through the linear projection $\pi _V: \PP^r\setminus V\to \PP^{2n}$ away from $V$. 
Write $E = S\setminus A$. Since $S$ is general both in $X$ and $\Gamma _q(X)$, we have $V\cap E=\emptyset$, $V\cap \langle E\rangle =\emptyset$. Note that $E\cap V =\emptyset$ implies $\Gamma _q(X)\nsubseteq V$. 

Since $S$ is general in $X$, $\pi _V(E)$ is a general subset of $X'$ with cardinality two; in particular, $\pi _V(E)\subset X'_{\reg}$. Since $\langle \cup _{p\in S} T_pX \rangle =\PP^r$, we have $\langle \cup _{p\in \pi _V(E)} T_pX'\rangle = \PP^{2n}$; one has $r_{\gen}(X')=2$. 
The linear projection $\pi_V$ restricted on $\Gamma_q(X)$ induces the linear projection $\pi_W$ (i.e. the linear projection from $W$) from $\Gamma _q(X)$ onto an irreducible curve $D$ of degree $d>1$. Applying $(s-2)$-many times Lemma \ref{v0}, we see that to prove the statement it is sufficient to prove that $D$ has $d=2$, i.e. it is a conic. 

Let $D^{\circ} = \pi _V\left(\Gamma _q(X)\setminus\left( \Gamma _q(X)\cap V\right)\right)$. Thus $D^{\circ}$ is an irreducible algebraic set of dimension $\le 1$ containing at least two points, i.e. the points in $\pi_V (E)$. The closure $D$ of $D^{\circ}$ is a curve containing $\pi_V (E)$. The following claim finishes off the proof: 
\begin{quote}
\noindent \emph{Claim :} $D$ is a plane curve of degree $d=2$.\\

\noindent \emph{Proof of the Claim :} One has $\dim \langle D\rangle \le 2$. Note that, since $X$ is of {\it type} {\bf IA}, we see that $D = \Gamma _z(X')$ for a general $z\in \PP^{2n}$ ($z$ depends on $E$). In other words, $D$ is an entry locus of $X'$. If $D$ were a line, then every point on its span $\langle D\rangle = D\subset X'$ would have $X'$-rank one, a contradiction. 

So the curve $D$ is not a line and hence $d\geq 2$. Since $S$ is general in $X$, $\pi_V (E)$ is general in $X'$ and spans a general $z\in \PP^{2n}$. The Trisecant lemma \cite[p. 109]{acgh} implies that $\langle \pi_V (E)\rangle \cap X'$ is a reduced degree two scheme. On the other hand, the scheme $\langle \pi_V (E)\rangle \cap X'$ contains the degree $d$ scheme $\langle \pi_V (E)\rangle \cap D$. Thus $d = 2$.
 \end{quote}
\end{proof}

In the next two results we do not need $(\dagger)$. 

\begin{lemma}\label{d1}
Let $X\subset \PP^4$ be an integral and non-degenerate curve. For a \emph{general} $p\in X_{\reg}$, set $L:= T_pX$. Let $\pi : \PP^4\setminus L\to \PP^2$ be the linear projection from $L$ and $f: X\to \PP^2$ the extension of $\pi _{X\setminus X\cap L}: X\setminus X\cap L \to \PP^2$. Assume that $f$ is not birational onto its image. Then $f(X)$ is not a conic. 
\end{lemma}

\begin{proof}
Let $\alpha  = \deg (f)\ge 2$ and $d = \deg (X)$. Assume that $f(X)$ is a conic. This is equivalent to $X$ being inside a quadric hypersurface $Q_L$ with vertex containing $L$. Since $X$ is non-degenerate, $Q_L$ is irreducible; $L$ is contained in its vertex and the latter cannot have dimension $\geq 2$. So $Q_L$ has vertex $L$.

Due to the generality assumption on $p$, for a general $o\in X_{\reg}$ with $L_o = T_o X$, the curve $X$ is contained in a quadric hypersurface $Q_{L_o}$, with vertex $L_o$. Since each $Q_{L_o}$ is irreducible, $Q_{L_o}$ and $Q_{L_{o'}}$ are distinct, if $L_o$ and $L_{o'}$ are distinct tangent lines. 

Let $U\subseteq X_{\reg}$ be the non-empty Zariski open subset of $X_{\reg}$ for which the quadric hypersurface $Q_{L_o}$ is defined. Set $\Phi = \cap _{o\in U} Q_{L_o}$. The tangential variety $\tau (X)$ of $X$ is a surface and $\tau(X)\subset \Phi$. Since $\Phi$ contains a surface containing $X$, 
it is non-degenerate and hence $\deg(\Phi)\geq 3$. Since $\Phi$ is contained in the intersection of two distinct quadrics, $\deg(\Phi)\leq 4$. Thus $\deg (\tau (X)) \in \{3,4\}$.  Let $H\subset \PP^4$ be a general hyperplane. Since $X \subseteq \mathrm{Sing}(\tau (X))$ (in characteristic zero), $\tau(X)\cap H$ is a non-degenerate integral space curve with at least $d\ge \alpha\cdot \deg(f(X))\geq 4\geq \deg(\tau(X))$ singular points, which is a contradiction.
\end{proof}

\begin{theorem}\label{d2}
Let $X\subset \PP^{3n+1}$ be an integral non-degenerate and non-defective variety of {\it type} {\bf IA} such that $\sigma _3(X)=\PP^{3n+1}$. Then $\Gamma _q(X)$ is a rational curve.
\end{theorem}

\begin{proof}
As in Theorem \ref{v1}, $\Gamma_q(X)$ is an irreducible curve. Fix a general $S\subset X$ such that $\sharp (S) =3$. We repeat the proof of Theorem \ref{v1}. In this case, $\sharp (A) = 1$, say $A =\{a\}$, with $a\in \Gamma _q(X)_{\reg}$ general. Using the notations above, the {\it Claim} in the proof of Theorem \ref{v1} shows that $f(\Gamma_q(X))$ is a smooth conic, independently of Lemma \ref{v0}. Note that $\ell(X) = 2s-2 = 4$, which means $\langle \Gamma_q(X)\rangle = \PP^4$. We apply Lemma \ref{d1} to $\Gamma_q(X)$, which yields that $f$ is birational onto its image. Thus $\Gamma_q(X)$ is a rational curve.
\end{proof}

\begin{small}

\end{small}


\begin{thebibliography}{99}

\bibitem{acgh} E. Arbarello, M. Cornalba, Ph. Griffiths and J. Harris, {\it Geometry of Algebraic Curves}, vol. I, Springer, Berlin - Heidelberg -
New York, 1985.

\bibitem{bp} M. Bolognesi and G. Pirola, {\it Osculating spaces and diophantine equations} (with
an Appendix by P. Corvaja and U. Zannier), Math. Nachr. 284 (2011), 960--972.

\bibitem{bp} M. Broadmann and E. Park,  {\it On varieties of almost minimal degree I: Secant loci of rational normal scrolls}, J. Pure Appl. Algebra 214 (2010), no. 11, 2033--2043. 

\bibitem{bs} M. Brodmann and P. Schenzel, {\it Arithmetic properties of projective varieties of almost minimal degree}, J. Algebraic Geom. 16 (2007), 347--400.

\bibitem{cac} A. Calabri and C. Ciliberto, {\it On special projections of varieties: epitome to a theorem of Beniamino Segre},
Adv. Geom. 1 (2001), 97--106.

\bibitem{cc3} L. Chiantini, and C. Ciliberto, {\it On the dimension of secant varieties},  J. Europ. Math. Soc. (JEMS) 73 (2006), no. 2, 436--454.

\bibitem{cc2} L. Chiantini, and C. Ciliberto, {\it On the concept of $k$-secant order of a variety}, J. London Math. Soc. 73 (2006), no. 2, 436--454.


\bibitem{fra} A. Franchetta, {\it Sulla curva doppia della proiezione di una superficie generale dello $S_4$ da un punto generico su
uno $S_3$}, Rend. R. Acc. D'Italia (1941); Rend. dell. Acc. dei Lincei 2 (1947), no. 8, 276--280; reprinted in : A. Franchetta,
Opere Scelte, Giannini, Napoli, 2006.

\bibitem{fl} W. Fulton and R. Lazarsfeld, {\it Connectivity and Its Applications in Algebraic Geometry}. Lecture
Notes in Mathematics, vol. 862 (Springer, New York, 1981), pp. 26--92.

\bibitem{fuji} T. Fujita, {\it Classification theories of polarized varieties}, London Mathematical
Society Lecture Notes Series 155, Cambridge University Press, 1990.

\bibitem{fujrob} T. Fujita and J. Roberts, {\it Varieties with small secant varieties: the extremal case}, Am. J. Math. {\bf 103}, 953--976, 1981. 

\bibitem{GH} J.P.~Griffiths and J.~Harris, {\it Principles of algebraic geometry}. Wiley-Interscience, John Wiley \& Sons, New York, 1978.

\bibitem{h} R. Hartshorne,  {\it Algebraic Geometry}. Springer, Berlin, 1977. 

\bibitem{he} J.~Harris, {\it Curves in projective space}, S\'eminaire de Math\'ematiques
  Sup\'erieures, vol.~85, Presses de l'Universit\'e de Montr\'eal, Montreal, Que., 1982, with the collaboration of D. Eisenbud. 
  
\bibitem{hsv} L-T. Hoa, J. St\"{u}ckrad, W. Vogel, {\it Towards a structure theory for projective varieties of degree $=$ codimension $+ 2$}, J. Pure Appl. Algebra 71 (1991) 203--231.

\bibitem{ir} P. Ionescu and F. Russo, {\it Defective varieties with quadratic entry locus, II}, Compositio Math. 144 (2009),
940--962.

\bibitem{j} J.-P. Jouanolou, {\it Th\'{e}or\`{e}mes de Bertini et applications}. Progress in Mathematics, 42. Birkh\"{a}user Boston, Inc., Boston, MA, 1983. 

\bibitem{k1} H. Kaji, {\it On the tangentially degenerate curves}, J. London Math. Soc. (2) 33 (1986), 430--440. 

\bibitem{k2} H. Kaji, {\it On the tangentially degenerate curves, II}, Bull. Braz. Math. Soc., New Series 45(4) (2014), 745--752.

\bibitem{lr} A. F. Lopez and Z. Ran, {\it On the irreducibility of secant cones and an application to linear normality}, Duke Math. J. 117 (2003), 389--401.

\bibitem{ma} A. Massarenti, {\it Generalized varieties of sums of powers}, Bull. Braz. Math. Soc., New Series 47({\bf 3}), 911--934. 

\bibitem{mm} A. Massarenti and M. Mella, {\it Birational aspects of the geometry of Varieties of Sums of Powers}, Adv. Math. {\bf 243}: 187--202, 2013. 

\bibitem{mp} E. Mezzetti and D. Portelli, {\it A tour through some classical theorems on algebraic surfaces}, An. \c{S}tiin\c{t} Univ.
Ovidius Constan\c{t}a Ser. Mat. 5 (1997), no. 2, 51--78.

\bibitem{Mukai1} S. Mukai, {\it Fano $3$-folds}, Lond.Math. Soc. Lect. Note Ser., {\bf 179}(1992),255--263.

\bibitem{Mukai2} S. Mukai, {\it Polarized K3 surfaces of genus $18$ and $20$}, Complex Projective Geometry, Lond. Math. Soc. Lect. Note Ser., Cambridge University Press, 1992, pp. 264--276. 

\bibitem{mum} D. Mumford, {\it Lectures on curves on an algebraic surface}. With a section by G. M. Bergman, Annals of Mathematics Studies, No. {\bf 59} Princeton University Press, Princeton, N.J. 1966.

\bibitem{nit} N. Nitsure, {\it Construction of Hilbert and Quot schemes}, Fundamental algebraic geometry, Math. Surveys Monogr., {\bf 123}, Amer. Math. Soc., Providence, RI, 2005. 

\bibitem{piene} R. Piene, {\it Cuspidal projections of space curves}, Math. Ann. {\bf 256}(1) (1981), 95--119.

\bibitem{pi} H. C. Pinkham, {\it A Castelnuovo bound for smooth surfaces}, Invent. Math. 83 (1986), 321--332.

\bibitem{rs} K. Ranestad and F.-O. Schreyer, {\it Varieties of sums of powers}, J. Reine Angew. Math., 525:147--181, 2000.

\bibitem{r2} F. Russo, {\it Defective varieties with quadratic entry locus, I}, Math. Ann. 344 (2009), 597--617. 

\bibitem{r3} F. Russo, {\it On the geometry of some special projective varieties}. Lecture Notes of the Unione Matematica Italiana, Springer, 2016. 

\bibitem{z} F. L. Zak, {\it Tangents and secants of varieties}. Translations of Mathematical Monographs {\bf {127}}, AMS, 1993.

\end{thebibliography}
\end{document}